%% file: proangle.tex
\documentclass[10pt]{article}

\usepackage{color}
\usepackage{graphicx}
\usepackage{amsmath}
\usepackage{amsmath}
\usepackage{amsxtra}
\usepackage{amssymb}
\usepackage{amsfonts}
\usepackage{amsthm}
\usepackage{amstext}
\usepackage{amsbsy}
\usepackage{amscd}
\usepackage[sans]{dsfont}

\theoremstyle{plain}
\newtheorem{theorem}{Theorem}

\newtheorem{lemma}{Lemma}

\newtheorem{proposition}[lemma]{Proposition}

\theoremstyle{definition}

\newtheorem{definition}[lemma]{Definition}

\theoremstyle{remark}
\newtheorem{example}[lemma]{Example}
\newtheorem{remark}[lemma]{Remark}

\def\defeq{:=}
\def\eqdef{=:}

\def\Xint#1{\mathchoice
{\XXint\displaystyle\textstyle{#1}}%
{\XXint\textstyle\scriptstyle{#1}}%
{\XXint\scriptstyle\scriptscriptstyle{#1}}%
{\XXint\scriptscriptstyle\scriptscriptstyle{#1}}%
\!\int}
\def\XXint#1#2#3{{\setbox0=\hbox{$#1{#2#3}{\int}$}
\vcenter{\hbox{$#2#3$}}\kern-.5\wd0}}

\def\dashint{\Xint-}
\newcommand{\avgint}{\dashint}

\newcommand{\topref}[2]{\overset{\text{\eqref{#1}}}{#2}}
\newcommand{\subeq}[2]{\mathord{\underbrace{\mathop{#1}}_{#2}}}
\newcommand{\supeq}[2]{\mathord{\overbrace{\mathop{#1}}^{#2}}}

\newcommand{\esssup}{\operatorname{esssup}}

\newcommand{\eqv}{\Leftrightarrow}
\newcommand{\impl}{\Rightarrow}

\newcommand{\selset}[2]{\{#1~:~#2\}}
\newcommand{\set}[1]{\{#1\}}

\newcommand{\R}{\mathds{R}}

\newcommand{\Rminus}{\R_-}

\newcommand{\boi}[2]{{]#1,#2[}}
\newcommand{\loi}[2]{{]#1,#2]}}
\newcommand{\roi}[2]{{[#1,#2[}}
\newcommand{\cli}[2]{{[#1,#2]}}

\newcommand{\pd}[1]{\partial_{#1}}
\newcommand{\pt}{\partial_t}

\newcommand{\ndiv}{\nabla\dotp}
\newcommand{\ncurl}{\nabla\crossp}
\newcommand{\nperp}{\nabla^\perp}
\newcommand{\tensor}{\otimes}
\newcommand{\hess}{\nabla^2}
\newcommand{\trace}{\operatorname{tr}}
\newcommand{\vlen}[1]{|#1|}

\newcommand{\half}{\frac12}

\newcommand{\dotp}{\cdot}
\newcommand{\crossp}{\times}

\newcommand{\csep}{\quad,\quad}

\newcommand{\Lap}{\Delta}

\newcommand{\Leb}{\mathcal L}
\newcommand{\Leba}[1]{\Leb^{#1}}
\newcommand{\Linf}{\Leba\infty}
\newcommand{\Lone}{\Leba1}

\newcommand{\defm}[1]{\emph{#1}}
\newcommand{\vv}{\vec v}
\newcommand{\graph}{\operatorname{graph}}
\newcommand{\const}{\text{const}}

\newcommand{\woba}[2]{\mathcal{W}^{#1,#2}} 

\newcommand{\eps}{\epsilon}

\newcommand{\upconv}{\nearrow}

\newcommand{\conv}{\rightarrow}

\newcommand{\Hoe}[2]{\spC^{#1,#2}}

\newcommand{\spC}{\mathcal{C}}
\newcommand{\Ck}[1]{\spC^{#1}}

\newcommand{\Cinf}{\Ck\infty}

\newcommand{\Czero}{\Ck0}
\newcommand{\Cone}{\Ck1}
\newcommand{\Ctwo}{\Ck2}

\newcommand{\DomR}{\Dom_R}
\newcommand{\cDomR}{\closure{\DomR}}
\newcommand{\cDomRi}{\cDomR\union\set{\infty}}

\newcommand{\bx}{\beta}

\newcommand{\Bv}{S}
\newcommand{\fs}{g}
\newcommand{\ffs}{\vec\fs}
\newcommand{\fz}{\fs^0}
\newcommand{\fo}{\fs^1}

\newcommand{\vsz}{\vs^0}
\newcommand{\vso}{\vs^1}

\newcommand{\QQ}{Q}
\newcommand{\TT}{\finrect_a}
\newcommand{\mirr}{M}

\newcommand{\Lama}{\Lam_a}
\newcommand{\lrva}{\lrv_a}
\newcommand{\lama}{\lam_a}

\newcommand{\Lamb}{\Lam_b}
\newcommand{\lrvb}{\lrv_b}

\newcommand{\union}{\cup}

\newcommand{\isect}{\cap}

\newcommand{\closure}[1]{\overline{#1}}
\newcommand{\bdry}{\partial}
\newcommand{\setdiff}{\backslash}

\newcommand{\norma}[2]{\norm{#1}_{#2}}
\newcommand{\norm}[1]{\left\|#1\right\|}
\newcommand{\snorm}[1]{\left|#1\right|}
\newcommand{\snorma}[2]{\snorm{#1}_{#2}}

\newcommand{\map}[3]{#1:#2\mapa#3}
\newcommand{\mapa}{\rightarrow}

\newcommand{\Det}{\operatorname{det}}

\newcommand{\sign}{\operatorname{sign}}
\newcommand{\cartp}{\times}

\newcommand{\finrect}{U}
\newcommand{\Dirt}{\Dir^\finrect}
\newcommand{\vu}{\vec u}
\newcommand{\Diri}{\Dir^\infty}

\newcommand{\carect}{R}
\newcommand{\Sigb}{\bdry\carect\isect\Slipb}%
\newcommand{\Sigi}{\bdry\carect\setdiff\Slipb}%

\newcommand{\polaeps}{\eps}
\newcommand{\gC}{C}

\newcommand{\rhsf}{F}

\newcommand{\rmin}{\underline\rad}
\newcommand{\lrvmin}{\underline\lrv}
\newcommand{\Pola}{\Theta}

\newcommand{\morex}{\delta}

\newcommand{\pLam}{\nabla_\Lam}

\newcommand{\Dir}{\mathcal{D}}

\newcommand{\quaf}{\quasic_\ff}
\newcommand{\ff}{\vec f}
\newcommand{\fg}{\vec F}

\newcommand{\isenc}{\gamma}

\newcommand{\polalo}{\pola^0}
\newcommand{\polahi}{\pola^1}
\newcommand{\polaa}{\pola^\lam}
\newcommand{\lrvlo}{\lrv_0}
\newcommand{\lrvhi}{\lrv_1}

\newcommand{\lam}{q}
\newcommand{\Lam}{\vec\lam}
\newcommand{\quasic}{K}
\newcommand{\elli}{\kappa}

\newcommand{\www}{w}

\newcommand{\wpot}{\Phi}

\newcommand{\sDom}{\Dom\union\Slipb} 

\newcommand{\Dom}{\Omega}

\newcommand{\Slipb}{B}
\newcommand{\vx}{v^x}
\newcommand{\vy}{v^y}
\newcommand{\lrv}{\ell}
\newcommand{\Lrv}{L}

\newcommand{\cz}{\overline z}
\newcommand{\pcz}{\pd\cz}
\newcommand{\pz}{\pd z}

\newcommand{\Dt}{D_t}
\newcommand{\vm}{\vec m}

\newcommand{\vmssonic}{\vms_1} 

\newcommand{\vort}{\omega}

\renewcommand{\Im}{\operatorname{Im}}

\newcommand{\dens}{\varrho}
\newcommand{\vpot}{\phi}

\newcommand{\stf}{\psi}

\newcommand{\piv}{p}
\newcommand{\pif}{\hat\piv}
\newcommand{\ipif}{\pif^{-1}}
\newcommand{\pp}{P}
\newcommand{\ppf}{\hat\pp}

\newcommand{\sobexf}{\sigma}

\newcommand{\rad}{r}
\newcommand{\vn}{\vec n}
\newcommand{\vs}{\vec s}
\newcommand{\xx}{{\vec x}}

\newcommand{\pola}{\theta}
\newcommand{\ssnd}{c}
\newcommand{\Mach}{M}

\newcommand{\ps}{\pd s}
\newcommand{\pl}{\pd\lrv}
\newcommand{\pr}{\pd\rad}

\newcommand{\px}{\pd x}
\newcommand{\py}{\pd y}

\newcommand{\vvi}{\vv(\infty)}

\newcommand{\gdiv}{\vec a}
\newcommand{\gmod}{\vec{\tilde a}}
\newcommand{\hdiv}{\hat\tau}
\newcommand{\hmod}{\tilde\tau}
\newcommand{\dhmod}{\tilde\tau'}
\newcommand{\dhdiv}{\hat\tau'}
\newcommand{\vms}{\mu}
\newcommand{\vmsmax}{\overline\vms}
\newcommand{\hexp}{\zeta}

\renewcommand{\vec}[1]{\mathbf{#1}}

\begin{document}

\newcommand{\mytitle}{Nonexistence of compressible irrotational inviscid flows along infinite protruding corners}
\title{\mytitle}%
\author{Volker Elling}%
\date{}
\maketitle%

\begin{abstract}%
  We consider inviscid flow with isentropic coefficient greater than one.
  For flow along smooth infinite protruding corners we attempt to impose a nonzero limit 
  for velocity at infinity at the upstream wall. 
  We prove that the problem does not have any irrotational uniformly subsonic solutions,
  whereas rotational flows do exist. 
  This can be considered a case of a slip-condition solid ``generating'' vorticity in inviscid flow. 
\end{abstract}

\section{Introduction and related work}

\begin{figure}[h]
  \centerline{\input{kj.pstex_t}}
  \caption{Flow around a body that is smooth except for one protruding corner}
  \label{fig:onecorner}
\end{figure}
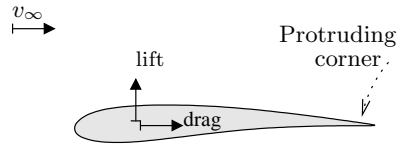

We consider steady planar flow around obstacles, one of the fundamental problems of fluid mechanics. 
Among its most important classical applications is the calculation of \defm{lift} (see fig.\ \ref{fig:onecorner}).
For flow around a bounded connected body with boundary smooth except for one \defm{protruding}\footnote{fluid-side corner angle greater than $\pi$, 
  as opposed to \defm{receding} corners which have angles less than $\pi$}
corner, somewhat accurate\footnote{see \cite[fig.\ 6.7.10]{batchelor} and the accompanying discussion}
and widely used formulas for lift can already be obtained at the level of incompressible irrotational inviscid flow: 
if we impose the \defm{Kutta-Joukowsky condition}, namely that $\vv$ is bounded, in particular at the corner, 
then there is a unique solution. 
In addition, formulas for the velocity field and hence the lift can be obtained for many shapes by basic complex analysis. 
The classical work of Frankl and Keldysh \cite{frankl-keldysh-1934}, Shiffman \cite{shiffman-exi-potf}, Bers \cite{bers-exi-uq-potf}, 
Finn-Gilbarg \cite{finn-gilbarg-uniqueness} etc.\ proves 
existence and uniqueness results and generalizes the Kutta-Joukowsky lift calculations to the compressible subsonic case.

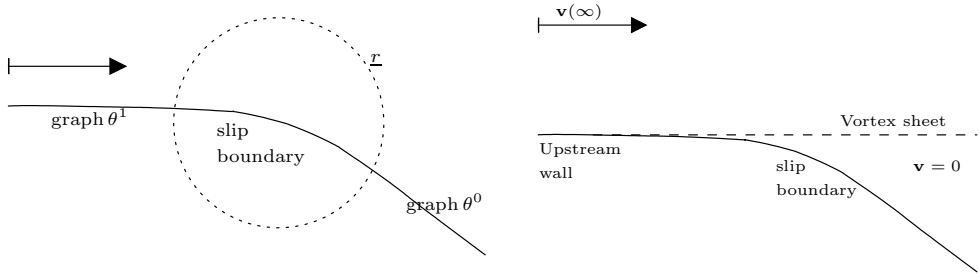
\begin{figure}[h]
  \hfil\parbox{2.6in}{\input{smoothangle.pstex_t}}\hfill\parbox{2in}{\input{smoothanglesheet.pstex_t}}\hfil
  \caption{Left: $\Dom$ covers an angle $\Pola$ at infinity, with variable but smooth boundaries. 
    This domain, for $\pp(\dens)=\dens^\isenc$ with $\isenc>1$,
    does not allow compressible uniformly subsonic flows whose velocity is bounded but nonvanishing at infinity. 
    Right: if vorticity is allowed, then there are trivial solutions with $\vv=\vvi=\const\neq 0$ above the vortex sheet, in particular
    near infinity at the upstream wall, whereas $\vv=0$ below the sheet.
  }
  \label{fig:infangle}
\end{figure}
Part of the well-known d'Alembert paradox (\cite{dalembert-paradox}, \cite[p.\ 54]{chorin-marsden}) 
is that irrotational flow --- with or without Kutta-Joukowsky condition --- predicts zero drag. 
Although this is a fair approximation for some shapes that do have very low drag, 
there are other shapes with significant drag even at small viscosity and small Mach number. 
To model nonzero drag it is necessary to allow solutions with \defm{vorticity},  
i.e.\ to go beyond purely irrotational flow. 
Such solutions are well-known \cite{helmholtz-1868,kirchhoff-wake-1869,rayleigh-1876,levi-civita-1907}, 
but it is necessary to explain how inviscid smooth flow can generate vorticity (or the closely related \defm{circulation}, which is needed for nonzero lift).

Prandtl's theory of boundary layers \cite{prandtl-1904} explains that thin boundary layers at low viscosity can be unstable, 
triggering flow separation and injecting vorticity into the interior of the fluid. 
While this is the observed physical mechanism \cite[fig.\ 34]{van-dyke}, it is worth pointing out that vorticity can be generated even if viscosity is exactly zero, 
in the sense that particular inviscid flow problems do not have \emph{any} irrotational solutions, but \emph{some} rotational ones. 

To this end, consider an infinite protruding smoothened or sharp corner (see fig.\ \ref{fig:infangle} left).
It is natural to look for solutions whose velocity is bounded and converges near the upstream wall ($\graph\polahi$) 
to a prescribed nonzero constant at spatial infinity. 
Indeed for \emph{super}sonic velocity many shapes have well-known solutions based on simple waves \cite[section 111]{courant-friedrichs}. 
However, we prove for polytropic pressure law with isentropic coefficient above $1$ 
that there do not exist any uniformly subsonic irrotational flows of this type. 
If we do permit rotation, then some of the same shapes do allow easily constructed solutions with straight vortex sheet separating from the wall 
(fig.\ \ref{fig:infangle} right). 

We make a point of proving nonexistence in the function class of (arguably) minimal regularity: 
essentially bounded velocities, without assumptions of continuity. 
For the wall we assume $\Ctwo$ unit tangents converging to constants
at a rate $o(|x|^{-\eps})$ as $|x|\conv\infty$, for arbitrarily small $\eps>0$ (see \eqref{eq:Domdef}).
We also prove an analogous result for incompressible flow (Theorem \ref{th:infangle-incomp}), 
assuming additionally that the velocity at infinity is bounded.
In the compressible case unbounded velocity is intrinsically impossible, since
the Bernoulli relation determines density as a function of velocity, and density reaches zero at a certain finite \defm{limit speed} (Section \ref{section:limitspeed})
and does not have any sensible definition at higher speeds. 

  \cite{yang-hui-huicheng-yin} show, under stronger regularity assumptions, 
  that the only subsonic flow in a smooth \emph{receding} corner is $\vv=0$. 
  Such results are sometimes called ``Liouville-type'' theorems. 
  This result
  has an alternative and very short proof by comparison principles \cite[Theorem 4.1]{elling-sepsheet},
  which are intrinsically insufficient to prove the protruding case. 

  However, our theorem is of a different nature: 
  example \ref{example:vnonzero} provides an incompressible flow around a \emph{smooth} protruding corner that has \emph{nonzero} velocity
  (vanishing only in the limit at infinity). 
  Protruding corners are fundamentally different from receding ones, as can be observed in Kutta-Joukowsky theory and many other aspects. 

  Protruding corners have also been studied extensively for minimal and capillary surfaces 
  (see \cite{danzhu-shi-pacj} and references therein). 
  The 2d minimal surface equation can be regarded as a special case\footnote{isentropic coefficient $-1$, sometimes called \defm{Chaplygin gas}}
  of plane compressible irrotational inviscid flow,
  but it is fundamentally different because the equation remains well-defined and even elliptic (albeit non-uniformly)
  as the ``velocity'' approaches infinity. 
  This is one example of several phenomena that are 
  considered unusual in actual fluids 
  and do not occur for  isentropic coefficients above $1$ (see Section \ref{section:limitspeed}).
  
  In section \ref{section:equations} we recall the necessary partial differential equations and fluid variables notation; 
  in section \ref{section:theorem} we state the precise theorem. 
  Section \ref{section:intslipregu} recalls basic regularity results for quasilinear elliptic PDE. 
  Section \ref{section:morrey} discusses in detail how to prove ``regularity at infinity'' which immediately yields the main result.

\section{Equations and main result}
\label{section:equations}

\subsection{Domain definition}

\begin{definition}
  \label{def:infangledom}%
    Consider polar coordinates $(\rad,\pola)$ and a set (see fig.\ \ref{fig:infangle} left)
    \begin{alignat}{5} \Dom = \selset{(\rad,\pola)}{\rad>\rmin,\ \polalo(\rad)<\pola<\polahi(\rad)} \label{eq:Domdef}\end{alignat} 
    where $0<\rmin<\infty$ and where $\polalo,\polahi$ are $\Ck3$ functions of $\rad\in\boi{\rmin}{\infty}$ satisfying 
    \begin{alignat}{5} |\pr^k\polaa| \leq C \rad^{-k-\polaeps} \qquad(\lam=0,1,~k=1,2,3) \label{eq:prpola}\end{alignat} 
    for some constants $C<\infty$ and $\polaeps>0$. This implies (by integrating the $k=1$ case over $\rad$) that
    \begin{alignat*}{5} \polaa(\rad) = \polaa(\infty) + O(\rad^{-\polaeps}) \quad\text{as $\rad\conv\infty$} \end{alignat*}
    The \defm{solid boundary} is 
    \begin{alignat*}{5} \Slipb = \selset{(\rad,\pola)}{\rad\in\boi{\rmin}{\infty},\ \pola\in\set{\polalo(\rad),\polahi(\rad)}} \end{alignat*} 
    its two components enclose an angle
    \begin{alignat*}{5} \Pola = \polahi(\infty)-\polalo(\infty) \end{alignat*} 
    which is assumed to be in $\boi0\pi$ or $\boi\pi{2\pi}$. 
\end{definition}

\subsection{Isentropic Euler}

The \defm{isentropic Euler} equations are
\begin{alignat*}{5} 
0 &= \pt\dens + \ndiv(\dens\vv) \\
0 &= \pt(\dens\vv) + \ndiv(\dens\vv\tensor\vv) + \nabla\pp 
\end{alignat*} 
where $\vv$ is velocity; pressure $\pp=\ppf(\dens)$ will be a strictly increasing $\Cinf$ function of density $\dens$. 

Assuming sufficient regularity the equations can be expanded into
\begin{alignat}{5} 
0 &= \Dt\dens + \dens\ndiv\vv \label{eq:masst}\\
0 &= \Dt\vv + \dens^{-1}\ppf_\dens(\dens)\nabla\dens \csep \Dt = \pt + \vv\dotp\nabla \label{eq:vpre}\end{alignat} 
Linearizing the equations around the constant solution 
\begin{alignat*}{5} \dens=\overline\dens=\const>0 \csep \vv=\overline\vv=\const=0 \end{alignat*}
yields 
\begin{alignat*}{5} 0 &= \pt\dens + \overline\dens \ndiv\vv \\
0 &= 
\pt\vv + \overline\dens^{~-1} 
\ppf_\dens(\overline\dens) \nabla\dens 
\end{alignat*} 
and subtracting $\ndiv$ of the lower equation from $\pt$ of the upper one yields
\begin{alignat*}{5} 
0 &= (\pt^2-\overline\ssnd^2\Lap)\dens 
\end{alignat*}
which is the linear wave equation with \defm{sound speed} 
\begin{alignat*}{5} \overline\ssnd \defeq \sqrt{\ppf_\dens(\overline\dens)} \end{alignat*} 

\eqref{eq:vpre} can be rewritten
\begin{alignat}{5} 0 = \Dt\vv + \nabla\piv  \label{eq:v}\end{alignat} 
where the \defm{enthalpy per mass}\footnote{Symbols like $h$ or $\pi$ are also commonly used; for consistency with the incompressible case we write $\piv$,
which should not be confused with pressure $\pp$.} $\piv=\pif(\dens)$ is defined (up to an additive constant) by
\begin{alignat}{5} \pif_\dens(\dens) = \dens^{-1}\ppf_\dens(\dens) \quad. \label{eq:pp-piv}\end{alignat}

\subsection{Potential flow}

Taking the curl of \eqref{eq:v} eliminates $\nabla\piv$, producing an equation for \defm{vorticity} $\vort=\ncurl\vv=v^y_x-v^x_y$:
\begin{alignat*}{5} 0 = \ncurl\pt\vv + \ncurl(\vv\dotp\nabla\vv) = ... = \Dt\vort + \vort\ndiv\vv \quad. \end{alignat*} 
Combined with \eqref{eq:masst} we obtain the transport equation
\begin{alignat*}{5} 0 = \Dt\frac\vort\dens\quad. \end{alignat*} 
If $\vort=0$ at $t=0$ (and $\dens>0$ throughout), then $\vort=0$ for all time.
There are important reasons to consider nonzero vorticity, 
as we pointed out in the introduction;
to demonstrate this, we explore consequences of assuming it is zero. 

$\ncurl\vv=0$ implies
\begin{alignat}{5} \vv = \nabla\vpot \label{eq:vpot}\end{alignat} 
for a scalar \defm{velocity potential} $\vpot$
(which is locally defined and may be multivalued when extended to non-simply connected domains). 

Henceforth we focus on stationary flow:
\begin{alignat}{5} 0 &= \ndiv(\dens\vv) \label{eq:mass} \quad, \\
0 &= \ndiv(\dens\vv\tensor\vv+\pp) \label{eq:mom}
\end{alignat} 
which yield 
\begin{alignat*}{5} 0 &= \vv\dotp\nabla\vv+\nabla\piv \quad. \end{alignat*} 
Into the latter substitute \eqref{eq:vpot} to obtain\footnote{with $\vv^2=\vv\vv^T$ and $\hess$ the Hessian operator}
\begin{alignat*}{5} 0 = \nabla\vpot^T\hess\vpot + \nabla(\pif(\dens)) = \nabla\big( \half|\nabla\vpot|^2 + \pif(\dens) \big) \quad. \end{alignat*} 
This implies the \defm{Bernoulli relation}\footnote{We will need only this special case.}
\begin{alignat}{5} \half|\vv|^2+\pif(\dens) = \text{Bernoulli constant}.  \label{eq:ber}\end{alignat} 
$\pif_\dens(\dens)=\dens^{-1}\ppf_\dens(\dens)=\dens^{-1}\ssnd^2>0$, so $\pif$ is strictly increasing. 
Hence we can solve for
\begin{alignat}{5} \dens = \ipif \big( \text{Bernoulli constant} - \half|\vv|^2 \big) \label{eq:pividens}\end{alignat} 
for some maximal interval of $|\vv|$ closed at its left endpoint $0$. 

Substituting \eqref{eq:pividens} into \eqref{eq:mass} yields a second-order scalar differential equation for $\vpot$ called \defm{compressible potential flow}. 
After differentiation it is equivalent to\footnote{with Frobenius product $A:B=\trace(A^TB)$; note $A:\vec w^2=\vec w^TA\vec w$}
\begin{alignat}{5} 0 = \big(I-(\frac{\vv}{\ssnd})^2\big):\hess\vpot = \big(1-(\frac{v^x}{\ssnd})^2\big)\vpot_{xx} - 2\frac{v^x}{\ssnd}\frac{v^y}{\ssnd}\vpot_{xy} 
+ \big(1-(\frac{v^y}{\ssnd})^2\big)\vpot_{yy} \label{eq:comp-potf}\end{alignat} 
where $\ssnd$ is a function of $\dens$, hence of $\vv=\nabla\vpot$. 
The eigenvectors of the coefficient matrix $I-(\vv/\ssnd)^2$ are $\vv$ and\footnote{$\perp$ counterclockwise rotation by $\pi/2$} $\vv^\perp$, 
with eigenvalues $1-\Mach^2$ and $1$ where 
\begin{alignat*}{5} \Mach \defeq \vlen\vv/\ssnd \end{alignat*} 
is the \defm{Mach number}. 
Hence \eqref{eq:comp-potf} is elliptic in a given point if and only if 
\begin{alignat*}{5} \Mach < 1 \quad, \end{alignat*} 
i.e.\ if and only if velocity $|\vv|$ is below the speed of sound $\ssnd$; such flows are called \defm{subsonic}. 

A \defm{uniformly subsonic} flow has $\Mach\leq 1-\delta$ for some constant $\delta>0$ independent of $\xx$. 
Many classical results have been extended to the non-uniformly subsonic case, 
but in this article we prefer brevity over a slight improvement in generality.

\subsection{Polytropic pressure law}
\label{section:limitspeed}%

For the remainder of the paper we avoid complications by focusing on polytropic pressure laws\footnote{for simplicity we suppress additional constant factors in $\ssnd,\dens,\piv$; they can always be made $1$ by changing physical units}:
\begin{alignat*}{5} \ppf(\dens) = \frac{\dens^\isenc}{\isenc} \end{alignat*} 
We only consider \defm{isentropic coefficients} $\isenc$ greater than $1$;
particularly important choices are $5/3$ (e.g.\ helium) and $7/5$ (e.g.\ air). 
Since the boundedness of velocity is a key point, 
we discuss the effects of increasing velocity in detail. 

The speed of sound is obtained via $\ssnd^2=\ppf_\dens(\dens)$ as 
\begin{alignat*}{5} \ssnd = \dens^{\frac{\isenc-1}2} \end{alignat*} 
which is positive for $\dens>0$. $\pif_\dens(\dens)=\dens^{-1}\ppf_\dens(\dens)$ yields, up to an additive constant, 
\begin{alignat}{5} \pif(\dens) = \frac{\dens^{\isenc-1}}{\isenc-1} \label{eq:pif-polytropic}\end{alignat} 
Using \eqref{eq:pif-polytropic}, \eqref{eq:pividens} becomes (by normalizing $\dens=1$ at $\vv=0$)
\begin{alignat*}{5} \dens =   \Big( 1 - \frac{\isenc-1}{2}|\vv|^2 \Big)^{\frac1{\isenc-1}} \quad. \end{alignat*} 
Since $\isenc>1$, $\dens$ decreases to \defm{vacuum} $0$ as $|\vv|$ increases to the finite \defm{limit speed} $\sqrt{2/(\isenc-1)}$. 
$\dens$ is \emph{undefined} above the limit speed, and there is no sensible way to modify the definition; 
informally speaking, fluid cannot\footnote{This is not true for arbitrary pressure laws, although there are some natural equivalent conditions.} accelerate to arbitrarily high speed by moving to near-vacuum regions
because the pressure and its gradient decay too rapidly to impart unbounded kinetic energy.

Finally
\begin{alignat*}{5} \ssnd^2 = \dens^{\isenc-1} = 1 - \frac{\isenc-1}{2}|\vv|^2 \end{alignat*} 
so 
\begin{alignat*}{5} \ssnd^2 - |\vv|^2 = 1 - \frac{\isenc+1}{2} |\vv|^2 \end{alignat*} 
The left-hand side is $>0$ if and only if the flow is subsonic ($|\vv|<\ssnd$), i.e.\ iff $|\vv|$ is below the \defm{critical speed} $\sqrt{2/(\isenc+1)}$.

\subsection{Streamfunction formulation}
\label{section:streamfunction}%

We will need an alternative formulation of irrotational flow, which is obtained as follows: $\ndiv(\dens\vv)=0$
implies\footnote{with $\nabla^\perp=(-\partial_y,\partial_x)$}
\begin{alignat*}{5} \dens\vv = -\nperp\stf \end{alignat*} 
for a scalar \defm{stream function} $\stf$. With $\vm=\rho\vv$,
consider the Bernoulli relation \eqref{eq:ber} in the form
\begin{alignat}{5} \text{Bernoulli constant} = \subeq{\supeq{\half|\vm|^2}{\vms}~\dens^{-2}+\pif(\dens) }{\eqdef \rhsf(\dens,\vms)} \quad \label{eq:berstf} \end{alignat} 
and apply the implicit function theorem, noting the trivial solution 
\begin{alignat*}{5} (\dens,\vms)=(1,0). \end{alignat*}
At solutions $(\dens,\vms)$ of \eqref{eq:berstf} 
that are vacuum-free and subsonic,
\begin{alignat*}{5} 
\frac{\partial\rhsf}{\partial\vms}
&=
\dens^{-2} > 0 \quad\text{and}
\\
\frac{\partial\rhsf}{\partial\dens}
&=
-\dens^{-3}|\vm|^2+\pif_\dens(\dens)
=
\dens^{-1}(c^2-|\vv|^2) > 0 \quad,
\end{alignat*} 
so we obtain a solution
\begin{alignat}{5} \frac1\dens = \hdiv(\vms) \label{eq:dens-hdiv}\end{alignat} 
for a strictly increasing function $\hdiv$ defined for $\vms$ in some maximal\footnote{This step works only for \emph{subsonic} velocities, so that the stream function formulation is awkward unless the flow is purely subsonic (or purely supersonic),
which is the only case we need.} interval $\cli{0}{\vmssonic}$ for some constant $\vmssonic\in\boi0\infty$; for $\vms=\vmssonic$ the velocity is exactly sonic.

Having solved the mass and Bernoulli equations it remains to ensure irrotationality\footnote{which is needed to recover the original velocity equation from the Bernoulli relation}:
\begin{alignat}{5} 0 = \ncurl \vv = \ncurl \frac{-\nperp\stf}{\dens} = -\ndiv\Big( \subeq{ \hdiv(\frac{|\nabla\stf|^2}{2}) \nabla\stf }{\eqdef g(\nabla\stf)} \Big) \label{eq:stf-divform} \end{alignat} 
\begin{definition}
  A \defm{uniformly subsonic compressible flow} in $\Dom$ is a \defm{stream function} $\stf\in\woba1\infty(\Dom)$
  (the space of functions with distributional derivatives that are essentially bounded functions), 
  satisfying
  \begin{alignat*}{5} \esssup_{\Dom} \half|\vm|^2 < \vmssonic \end{alignat*} 
  (which is equivalent to $\esssup_{\Dom}\Mach < 1$),
  and satisfying \eqref{eq:stf-divform} in the distributional sense, i.e.
  \begin{alignat*}{5} 0 = \int_{\Dom} \hdiv(\half|\nabla\stf|^2) \nabla\stf\dotp\nabla\vartheta~d\xx \end{alignat*} 
  for every smooth function $\vartheta$ with compact support in $\Dom$,
  as well as the slip condition\footnote{for our domains $\woba1\infty$ has well-defined trace on the boundary} $\stf=0$ on $\Slipb$.
\end{definition}

Assuming sufficient additional regularity, differentiation yields after some calculation that
\begin{alignat}{5} 
0
&=
\big(1-(\frac{\vv}{\ssnd})^2\big):\hess\stf = \big(1-(\frac{v^x}{\ssnd})^2\big)\stf_{xx} - 2\frac{v^xv^y}{\ssnd^2}\stf_{xy} + \big(1-(\frac{v^y}{\ssnd})^2\big)\stf_{yy} 
\label{eq:stf-2d}
\end{alignat} 
which has the same coefficient matrix as \eqref{eq:comp-potf}; again it is elliptic if and only if the flow is subsonic.

\subsection{Incompressible flow}

The incompressible limit of \eqref{eq:stf-2d} is obtained by (for example) considering sequences of solutions with velocities approaching $0$
(and hence sound speed converging to a positive constant), 
This yields the limit
\begin{alignat*}{5} 0 &= -\Lap\stf \end{alignat*}
which is \eqref{eq:stf-divform} with $\hdiv=1$. 
This can also be obtained (see e.g.\ \cite{klainerman-majda-singular}) along similar lines as for compressible flow from the unsteady incompressible Euler equations 
\begin{alignat*}{5} 0 &= \ndiv\vv, \\
0 &= \Dt\vv + \nabla\piv;
\end{alignat*} 
here $\dens=\const$, and $\piv$ is not a function of $\dens$ but rather a separate unknown making the second equation divergence-free. 
We can require $\stf\in\woba1\infty(\Dom)$ for consistency, 
but standard regularity results for harmonic functions will yield analyticity in the interior $\Dom$ anyway. 
The Mach number of incompressible solutions is generally ``defined'' as $0$, which is also the actual limit in various definitions of ``incompressible limit''.

Incompressible potential flows correspond to harmonic functions; 2d harmonic functions are conveniently represented by holomorphic functions of a single variable. 
To this end it is customary to consider the \defm{complex velocity}
\begin{alignat*}{5}
  \www &\defeq \vx-i\vy 
\end{alignat*}
as a function of
\begin{alignat*}{5}  
z &\defeq x+iy. 
\end{alignat*} 
Then 
\begin{alignat*}{5}
  \pcz \www 
  &= \half(\px+i\py)(\vx-i\vy) 
  = \half\big( \px\vx+\py\vy + i(\py\vx-\px\vy) \big) 
  = \half( \ndiv\vv - i\ncurl\vv ) .
\end{alignat*}
Hence $\www$ represents an \emph{incompressible} and \emph{irrrotational} flow if and only if $\www$ is holomorphic. 

If so, it is convenient to use the \defm{complex velocity potential} $\wpot=\int^z\www~dz$; the lower endpoint of the integral is fixed (changing it only adds a constant);
the path does not matter locally since $\www$ is holomorphic, but for non-simply connected domains $\wpot$ may be multivalued
(the simplest example being the \defm{point vortex} $\wpot=\frac1{2\pi i}\log z$ which has multi-valued $\vpot$,
but corresponds to the single-valued $\vv=(2\pi)^{-1}|\xx|^{-2}(-y,x)$).
\begin{alignat*}{5}
  \wpot = \vpot + i\stf 
\end{alignat*}
is also holomorphic, satisfying Cauchy-Riemann equations
\begin{alignat*}{5}
  \vpot_x = \stf_y \csep \vpot_y = -\stf_x ,
\end{alignat*}
and hence 
\begin{alignat*}{5}
  \www 
  &= \vx&&-i\vy 
  &&= \pz\wpot 
  = \half(\px-i\py)(\vpot+i\stf) 
  = \half[(\vpot_x+\stf_y) + i(\stf_x-\vpot_y)]
  \\&= \vpot_x &&- i\vpot_y 
  \\&= \stf_y &&- i(-\stf_x)
  \\\eqv\quad \vv &= \begin{bmatrix}
    \vx  \\ \vy
  \end{bmatrix} &&= \nabla\vpot = -\nperp\stf
\end{alignat*}

\subsection{Slip condition}

At solid boundaries we use the standard \defm{slip condition}
\begin{alignat}{5} 0 = \vn \dotp \vv\quad, \label{eq:slip}\end{alignat} 
where $\vn$ is a normal to the solid. For potential flow:
\begin{alignat*}{5} 0 = \vn \dotp \nabla\vpot \quad. \end{alignat*} 
In the stream function formulation:
\begin{alignat*}{5} 0 = \vs \dotp \nabla\stf \quad, \end{alignat*} 
where $\vs$ is a tangent to the solid. In the latter case integration along connected components of (say) a piecewise $\Cone$ boundary yields 
\begin{alignat}{5} \stf = \const \label{eq:stf-const} \quad. \end{alignat} 
For incompressible flow we may also use the complex velocity potential formulation:
\begin{alignat*}{5}
  \Im\wpot = \const ,
\end{alignat*}
i.e.\ given $\wpot$ the level sets of $\Im\wpot$ may serve as solid boundaries.

If the solid boundary has a single connected component, 
we may add an arbitrary constant to $\stf$ without changing $\vv=-\nperp\stf$ to obtain the convenient zero Dirichlet condition
\begin{alignat*}{5}
  \stf = 0.
\end{alignat*}
In some important cases, such as flow through a nozzle, the solid boundary has several connected components, 
e.g.\ lower and upper nozzle boundary, and we generally need several different constants which cannot all be zero. 
Although we do have two components when considering only a domain $\Dom$ at infinity, 
our physical perspective is flow along smooth protruding corners, hence simply connected domains, with connected boundary, so we use $\stf=0$.

  \section{Theorem}
  \label{section:theorem}

  \begin{theorem}
    \label{th:infangle}%
    If $\Pola\notin\set{0,\pi,2\pi}$, 
    and for polytropic pressure law with isentropic coefficient above $1$, 
    there do not exist any uniformly subsonic flows in $\Dom$ whose velocity does not vanish at infinity. 
  \end{theorem}
  The proof is completed at the end of section \ref{section:morrey}. 
  For incompressible flow there is an analogous theorem (which is at least partly folklore):
  \begin{theorem}
    \label{th:infangle-incomp}%
    If $\Pola\notin\set{0,\pi,2\pi}$, 
    there do not exist incompressible flows in $\Dom$ whose velocity \emph{is bounded} but does not vanish at infinity. 
  \end{theorem}

  \begin{figure}
    \hfil\input{vnonzero.pstex}\hfil%
    \caption{Boundary ($\stf=0$, thick) and streamlines ($\stf>0$) for $\wpot=i(z^{\frac23}-1)$;
      the domain covers an angle $\Pola=\frac32\pi$ at infinity, where $|\vv|\sim\rad^{-\frac13}\conv 0$.
    }
    \label{fig:vnonzero}
  \end{figure}
  \begin{example}
    \label{example:vnonzero}%
    Even in a globally defined domain the stronger statement ``$\vv=0$ everywhere'' is generally false.
    For example consider the complex potential
    \begin{alignat*}{5} \wpot(z) = i(z^{\frac23}-1) \end{alignat*} 
    with fractional power branch cut on $\Rminus$.
    With $z=\rad e^{i\pola}$,
    $\wpot$ yields a domain
    \begin{alignat*}{5} \Dom = \set{\Im\wpot>0} 
    = \set{\rad>\cos(\frac23\pola)^{-\frac32}}. \end{alignat*} 
    $\cos(\frac23\pola)^{-\frac32}$ is smooth and positive for $\pola\in\boi{-\frac34\pi}{\frac34\pi}$;
    it converges to $+\infty$ as $\pola$ approaches the endpoints of the interval. Hence $\rad$ is lower-bounded away from $0$,
    and therefore the complex velocity $\www=\wpot'$ is bounded:
    \begin{alignat*}{5} |\wpot'(z)| = |i \frac23 z^{-\frac13}| \sim \rad^{-\frac13} . \end{alignat*} 
  \end{example}

\section{Interior and slip boundary regularity}
\label{section:intslipregu}

As discussed in section \ref{section:limitspeed} and after \eqref{eq:dens-hdiv}, 
for $\isenc>1$ uniformly subsonic flow must satisfy a gradient bound
\begin{alignat*}{5} \half|\nabla\stf|^2 \leq \vmsmax < \vmssonic < \infty \end{alignat*} 
for some constant $\vmsmax$ which depends only on $\isenc$ and $\esssup\Mach<1$.

Consider our equation in divergence form:
\begin{alignat*}{5} 0 \topref{eq:stf-divform}{=} -\ndiv \gdiv(\nabla\stf) = -\gdiv_\vm(\nabla\stf) : \nabla^2\stf . \end{alignat*} 
$\gdiv_\vm(\vm)$ is not necessarily positive definite for large $|\vm|$. 
We need to use comparison principles which require ellipticity, i.e.\ positive definite $\gdiv_\vm$.

In our case
\begin{alignat*}{5} \gdiv(\vm) = \hdiv(\supeq{\frac{|\vm|^2}{2}}{\vms})\vm, \end{alignat*} 
so it is convenient to use the following general cutoff method:
\begin{alignat*}{5} 
\gmod(\vm)
\defeq 
\hmod(\frac{|\vm|^2}{2})\vm
\end{alignat*} 
where
\begin{alignat*}{5}
\hmod(\vms) 
\defeq 
\hdiv(\vms)
\quad\text{for $\vms \leq \vmsmax $.} 
\end{alignat*}
Then
\begin{alignat*}{5} \gmod_\vm(\vm) &= \hmod(\half|\vm|^2)I + \dhmod(\half|\vm|^2)\vm^2 ; \end{alignat*} 
the eigenvectors of $\gmod_\vm$ are $\vm^\perp$ and $\vm$, with eigenvalues $\hmod(\vms)$ and $\hmod(\vms)+2\vms\dhmod(\vms)$, 
so positive definiteness is equivalent to $\hmod> 0$ and
\begin{alignat}{5} 
0 < \hmod + 2\vms \dhmod .  \label{eq:ddd}
\end{alignat} 
Since we need $\hmod$ to be $\Cone$, the straightforward choice $\hmod(\vms)=\hmod(\vmsmax)=\const>0$ for $\vms>\vmsmax$ is insufficient. 
Instead, we consider the equivalent form
\begin{alignat*}{5} 
-\half < \frac{\partial\log\hmod}{\partial\log\vms}.
\end{alignat*} 
We already have positive definiteness for $\vms\leq\vmsmax$, so 
\begin{alignat*}{5} -\half < \hexp \defeq (\frac{\partial\log\hmod}{\partial\log\vms})_{|\vms=\vmsmax}. \end{alignat*} 
It is sufficient to solve for $\vms>\vmsmax$ the ODE
\begin{alignat*}{5} \hexp = \frac{\partial\log\hmod}{\partial\log\vms} \end{alignat*} 
with initial condition $\hmod(\vmsmax)=\hdiv(\vmsmax)$ which then implies by choice of $\hexp$ that $\dhmod(\vmsmax)=\dhdiv(\vmsmax)$ as well. 
Solution: 
\begin{alignat*}{5} \hmod(\vms) = \vms^{\hexp} \vmsmax^{-\hexp}\hmod(\vmsmax) \end{alignat*}

\begin{proposition}
  \label{prop:base-regularity}%
  Compressible uniformly subsonic irrotational flows in $\Dom$ satisfy
  \begin{alignat*}{5} \stf \in \Hoe2\sobexf(\sDom) \end{alignat*} 
  for some $\sobexf\in\boi01$.
\end{proposition}
\begin{proof}
  \newcommand{\otherstf}{\tilde\stf}
  Using the truncated $\gmod$ version of $\gdiv$, 
  the divergence form \eqref{eq:stf-divform} of our equation yields
  \begin{alignat}{5} 0 = \ndiv(\gmod(\nabla\stf)) \label{eq:truncdiv}\end{alignat} 
  \begin{enumerate}
  \item
    First consider interior points $\xx\in\Dom$. 
    Choose a small open ball $G$ centered in $\xx$ contained in $\Dom$. 
    We prove regularity of $\stf$ by 
    showing 
    existence of a function $\otherstf\in\Hoe2\sobexf(G)\isect\Czero(\closure G)$ 
    for some $\sobexf\in\boi01$ that solves the same elliptic boundary-value problem, 
    then arguing uniqueness for that problem. 
    
    \cite[Theorem 12.5]{gilbarg-trudinger} yields existence of $\otherstf$ solving
    \begin{alignat}{5} 
    0 &= \ndiv\big(\gmod(\nabla\otherstf)\big) \quad\text{in $G$,} \label{eq:truncdivother}\\
    \stf &= \otherstf \quad\text{on $\bdry G$.}\notag
    \end{alignat} 
    To show that $\stf=\otherstf$, we use a comparison principle 
    which can be found in similar but not quite applicable forms 
    throughout the literature (see \cite[Theorem 10.7]{gilbarg-trudinger}): 
    Take the difference of the two equations and abbreviate $d=\stf-\otherstf$:
    \begin{alignat}{5} 0 &= -\ndiv\big(\gmod(\nabla\stf)-\gmod(\nabla\otherstf)\big)
      \notag\\&= -\ndiv\Big(\int_0^1 \frac{d}{dt}\big(\gmod((1-t)\nabla\otherstf+t\nabla\stf)\big)dt\Big) 
      \notag\\&= -\ndiv\Big(\subeq{\int_0^1 \gmod_\vm\big((1-t)\nabla\otherstf+t\nabla\stf)\big)dt}{\eqdef\overline A} \nabla d \Big)
      \label{eq:othercomp}
    \end{alignat} 
    Since $\nabla\otherstf$ and $\nabla\stf$ are essentially bounded, their convex combinations are (essentially) in a compact set
    on which $\gmod_\vm$ is by continuity \emph{uniformly} positive definite. Hence $\overline A$ is positive definite uniformly in $\xx$. 

    $\stf,\otherstf$ and hence $d$ are $\woba{1}{\infty}(G)$, 
    so $\overline A\nabla d\in\Linf(G)$. 
    That means $d$ may be used as test function in \eqref{eq:othercomp}: $\stf,\otherstf$ and hence $d=\stf-\otherstf$ are in $\woba1\infty(G)$,
    and $d=0$ on $\bdry G$ where $\otherstf=\stf$ by definition, 
    so we can find smooth test functions with compact support in $G$ whose gradients converge in $\Lone(G)$ to $\nabla d$. 
    Thus 
    \begin{alignat*}{5} 0 = \int_\Dom \nabla d\dotp \overline A(\xx) \nabla d~d\xx \end{alignat*} 
    which implies $\nabla d=0$, hence $d$ constant and thus $d=0$, since $d=0$ on $\bdry G$,
    so that $\stf=\otherstf$ in $\closure G$ which implies the desired regularity for $\stf$.
  \item
    To show regularity at $\Slipb$ we may use a standard boundary Schauder estimate,
    for example \cite[Lemma 6.18]{gilbarg-trudinger} using that (in particular) $\stf\in\Czero(\sDom)$
    by definition
    and $\stf\in\Ctwo(\Dom)$ by the previous step.
  \end{enumerate}
\end{proof}
\noindent This regularity result allows using the non-divergence form \eqref{eq:stf-2d} of compressible potential flow,
using point values of $\stf$, replacing $\esssup$ with $\sup$ etc, which we do henceforth without further mention.

\section{Regularity at infinity}
\label{section:morrey}

There is a large literature on showing regularity of solutions of elliptic PDE at domain corners; 
each method has some limitations. 
Classical work \cite{bers-exi-uq-potf,finn-gilbarg-uniqueness} is based on the theory of quasiconformal maps. 
The proofs implicitly use Riemann mapping theorem, pseudo-analytic functions and other tools from complex analysis,
some of which were not adopted widely or superseded by more recent\footnote{%
  The classical work predates the breakthroughs of de Giorgi and Nash.} 
developments. 
Here we give a self-contained proof which can handle somewhat more general boundary conditions as well.

For $R>\rmin$ define the following neighbourhoods of infinity:
\begin{alignat*}{5} \DomR \defeq \selset{(\rad,\pola)}{ r\in\boi{R}{\infty},~\pola\in\boi{\polalo(\rad)}{\polahi(\rad)} }, \\
\cDomR \defeq \selset{(\rad,\pola)}{ r\in\roi{R}{\infty},~\pola\in\cli{\polalo(\rad)}{\polahi(\rad)} }. \end{alignat*} 
$\Hoe0\morex(\cDomRi)$ with $\morex\in\boi01$ is the Banach space of continuous functions on $\DomR$ with well-defined limits 
in any element of $\cDomRi$ and with finite norm
\begin{alignat*}{5} \norma{u}{\Hoe0\morex(\cDomRi)} = \norma{u}{\Linf(\DomR)} + \snorma{u}{\Hoe0\morex(\cDomRi)}  \end{alignat*} 
where
\begin{alignat*}{5} \snorma{u}{\Hoe0\morex(\cDomRi)} = \sup_{\xx_1,\xx_2\in\DomR,~\xx_1\neq\xx_2}  |u(\xx_1)-u(\xx_2)| 
|\xx_1-\xx_2|^{-\morex} \min\{|\xx_1|,|\xx_2|\}^{2\morex}
\end{alignat*}
(finiteness implies $\morex$-H\"older continuity, including ``at infinity'', in the sense 
of standard H\"older continuity after a coordinate change $\xx\mapsto|\xx|^{-2}\xx$). 

\begin{theorem}
  \label{th:inflocalmorrey}%
  Consider uniformly subsonic flows $\stf\in\Ctwo(\Dom\union\Slipb)$.
  For $R<\infty$ sufficiently large,
  \begin{alignat*}{5} \norma{\nabla\stf}{\Hoe0\morex(\cDomRi)} \leq \gC, \end{alignat*} 
  where $R$, $\morex\in\boi01$ and $\gC<\infty$ 
  depend only on $\Dom$, $\isenc$ and $\esssup\Mach<1$. 
\end{theorem}
\begin{remark}
  The result also holds if $\Dom,\DomR$ are neighbourhoods of infinity, by substituting slip conditions with periodic boundary conditions
  in the following discussion,
  with obvious modifications. 
\end{remark}

\paragraph{Quasiconformal maps}

To prove this theorem we 
begin by reviewing basics about quasiconformal maps. While there is an extensive theory \cite{ahlfors-lectures-quasiconformal}, 
we only use the part relevant to quasilinear elliptic PDE in the plane, based on the classical work of Morrey \cite{morrey-1938} and subsequent authors
(see \cite[Chapter 12 and section 13.2]{gilbarg-trudinger} for historical remarks). 

For a matrix $A$, we denote by $|A|$ the operator norm induced by the Euclidean norm on vectors.
\begin{definition}
  Let $U\subset\R^2$ be an open set. 
  A $\Cone$ map $\map{\ff}{U}{\R^2}$ is called \defm{quasiconformal}\footnote{%
    A \defm{conformal} map satisfies \eqref{eq:quasiconf} with $\gC=1$; 
    in that case $\Det\ff'$ is called \defm{conformal factor}. 
    The definition implies $\Det\geq 0$ ($\Det$ can touch $0$ but not change sign altogether), 
    i.e.\ $\ff$ is orientation-preserving, which is not essential here.
    Other texts use the Frobenius norm instead of the Euclidean operator norm; 
    they are equivalent and we are not concerned with optimal values of $\gC$ or $\morex$. 
  } on $U$
  if there is a constant $\quaf\in\boi0\infty$ so that for all $\xx\in U$ 
  \begin{alignat}{5} |\ff'(\xx)|^2 \leq \quaf \Det\ff'(\xx)   \label{eq:quasiconf} \end{alignat} 
\end{definition}

\begin{proposition}
  \label{prop:composition-quasi}%
  The composition of quasiconformal maps is quasiconformal. 
\end{proposition}
\begin{proof}
  Let $\fg$ be quasiconformal on $U$, $\ff$ quasiconformal on a  superset of $\fg(U)$. Then
  \begin{alignat*}{5} |(\ff\circ\fg)'|^2 &= |(\ff'\circ\fg)\fg'|^2 \leq |\ff'\circ\fg|^2 |\fg'|^2
  \leq \quaf \Det(\ff'\circ\fg) \quasic_\fg \Det\fg'
  \\&= \quaf\quasic_\fg \Det((\ff\circ\fg)')
  \end{alignat*} 
  so $\ff\circ\fg$ is quasiconformal as well.
\end{proof}

The relationship between quasiconformal maps and elliptic PDE is established by the following observation. 
Let 
\begin{alignat*}{5} \mirr = \begin{bmatrix}
  1 & 0 \\
  0 & -1
\end{bmatrix}. \end{alignat*}
\begin{proposition}
  \label{prop:quasiconformal-elliptic}%
  If a $\Ctwo$ map $\map{u}{U}{\R}$ solves a scalar second-order PDE 
  \begin{alignat*}{5} 0 = A(\xx):\hess u \quad\text{in $U$} \end{alignat*} 
  which is uniformly elliptic, then the map $\ff=M\nabla_\xx u$ is quasiconformal. 
\end{proposition}
\begin{proof}
  Consider any $\xx\in U$. 
  $\hess u$ is real symmetric, so 
  \begin{alignat*}{5} \hess u = \sum_{i=1,2} \lambda_i v_i^2 \end{alignat*} 
  for real orthonormal eigenvectors $v_i$ and real eigenvalues $\lambda_i$. Hence
  \begin{alignat*}{5} 0 =& A:\hess u
  = 
  \sum_{i=1,2} \lambda_i v_i^T A v_i 
  \\\eqv\quad
  & 
  -\lambda_1 \frac{v_1^T Av_1}{v^T_2Av_2} = \lambda_2 \end{alignat*} 
  The fraction is in $\cli{\elli}{1/\elli}$, where $\elli\in\loi01$ is the ellipticity, 
  so $\sign\lambda_2=-\sign\lambda_1$ and\footnote{The argument fails in dimensions three and higher, since
    $0=\lambda_1 v_1\dotp Av_1 + ... + \lambda_n v_n\dotp Av_n$
    requires only the largest two of $n$ eigenvalue magnitudes to be comparable.
  }
  \begin{alignat*}{5} 
  |\lambda_1| \leq \frac1\elli |\lambda_2| 
  \csep
  |\lambda_2| \leq \frac1\elli |\lambda_1| 
  \end{alignat*} 
  Therefore
  \begin{alignat*}{5} |\hess\stf|^2 = \max\{|\lambda_1|^2,|\lambda_2|^2\} \leq \frac1\elli |\lambda_1||\lambda_2| = - \frac1\elli \Det\hess\stf
  \overset{\Det M=-1}{=} \frac1\elli \Det(\nabla_\xx\ff). \end{alignat*} 
\end{proof}

\newcommand{\gdel}{\delta}

\begin{proof}[Proof of Theorem \ref{th:inflocalmorrey}]
  Our $\stf$ solves
  \begin{alignat*}{5} 0 \topref{eq:stf-2d}{=} \big(I-(\frac{\vv}{\ssnd})^2\big):\nabla^2\stf \end{alignat*} 
  which is uniformly elliptic since we assumed uniformly subsonic $\stf$.
  Hence
  \begin{alignat*}{5} |\ff|^2 \leq \gC \Det\ff \quad\text{on $\Dom$;} \end{alignat*} 
  here and henceforth constants $\gC\in\boi0\infty$ may depend on $\Dom$, $\isenc$ and $\esssup\Mach$ but nothing else;
  each instance of $\gC$ may be a different constant.

  \paragraph{Change to log-polar and straight strip}

  For the following it is convenient to change to log-polar coordinates $\Lrv=(\lrv,\pola)$ where $\lrv = \log\rad$,
  and then to define a a new coordinate $\lam\in\boi01$ by
  \begin{alignat*}{5} \pola = \lam\polahi + (1-\lam)\polalo \end{alignat*} 
  so that coordinates $\Lam=(\lrv,\lam)$ lie in the \emph{straight}-boundary semi-infinite strip 
  $\boi{\lrvmin}{\infty}\cartp\boi01$. We will increase $\lrvmin<\infty$ a finite number of times,
  with ultimate value depending only on $\Dom$, $\isenc$ and $\esssup\Mach$.
  
  $\Lrv\mapsto\xx$ is a conformal change of coordinates, so by Proposition \ref{prop:composition-quasi} 
  $\Lrv\mapsto\ff$ is again quasiconformal. 
  
  The Jacobian of $\Lam=(\lrv,\lam)\mapsto(\lrv,\pola)=\Lrv$ is 
  \begin{alignat*}{5} \pLam\Lrv = \begin{bmatrix}
    1 & 0 \\
    \lam\pl\polahi+(1-\lam)\pl\polalo & \polahi-\polalo
  \end{bmatrix}. \end{alignat*} 
  The assumption \eqref{eq:prpola} about the boundary yields for $k=1,2,3$ that 
  \begin{alignat}{5} |\pl^k\polaa| = |(\rad\pr)^k\polaa| \leq C\sum_{j=1}^k \rad^j|\pr^j\polaa| 
  \topref{eq:prpola}{\leq} 
  Ce^{-\polaeps\lrv} .
  \label{eq:pl-polaa} \end{alignat} 
  In particular integration of the $k=1$ case yields that 
  \begin{alignat*}{5} \polaa = \const + O(e^{-\polaeps\lrv}) \quad(\lrv\conv\infty) \end{alignat*} 
  so that $\polahi-\polalo$ is continuous at infinity, with $\Pola=\polahi(\infty)-\polalo(\infty)>0$ by assumption. 
  Hence the Jacobian converges to 
  \begin{alignat*}{5} \begin{bmatrix}
    1 & 0 \\
    0 & \Pola
  \end{bmatrix} \end{alignat*} 
  which is invertible, 
  so for $\lrv$ in $\boi{\lrvmin}{\infty}$ with $\lrvmin<\infty$ sufficiently larger
  $\Lam\mapsto\Lrv$ is quasiconformal, and by composition 
  (Proposition \ref{prop:composition-quasi}) $\Lam\mapsto\ff$ is quasiconformal.

  \paragraph{Boundary-adapted quasiconformal map}

  We pass from $\ff$ to another quasiconformal map $\ffs$ that corresponds directly\footnote{This trick is similar to \cite{lieberman-pacj-1988} which provides another approach to corner regularity.} to the boundary conditions. 
  For $i=0,1$ we define a vector field $\vs^i$ tangent to the $\lam=i$ half of the slip boundary $\Slipb$, more precisely
  \begin{alignat*}{5} \vs^i(\lrv,\lam) 
  \defeq
  \begin{bmatrix}
    \cos\polaa(\lrv) \\
    \sin\polaa(\lrv)
  \end{bmatrix} 
  + 
  \pl\polaa(\lrv)
  \begin{bmatrix}
    -\sin\polaa(\lrv) \\
    \cos\polaa(\lrv)
  \end{bmatrix} .
  \end{alignat*}
  Let $\Bv$ be the matrix with columns $\mirr\vsz,\mirr\vso$;
  by $\pl\polaa\conv 0$ as $\lrv\conv\infty$ we have $\Det\Bv=-(\cos\polalo,\sin\polalo)\crossp(\cos\polahi,\sin\polahi)=-\sin\Pola>0$. 
  Our assumptions \eqref{eq:prpola} yield, after increasing $\lrvmin$ if necessary, that $\Bv$ is invertible for $\lrv\in\roi{\lrvmin}{\infty}$ with
  \begin{alignat*}{5} |\Bv|,|\Bv|^{-1}\leq \gC \csep |\pLam\Bv|,|\pLam(\Bv^{-1})| \topref{eq:pl-polaa}{\leq} \gC e^{-\polaeps\lrv} . \end{alignat*} 
  Set
  \begin{alignat*}{5} \ffs = (\fz,\fo) \defeq \Bv\ff, \end{alignat*} 
  then we achieve the main purpose of $\ffs$: for $i\in\set{0,1}$
  \begin{alignat*}{5} \fs^i &= (\mirr\vs^i)\dotp(\mirr\nabla\stf) = \vs^i\dotp\nabla\stf = 0 \quad\text{on the $\lam=i$ part of $\Slipb$.} \end{alignat*}
  $\ffs$ satisfies 
  \begin{alignat*}{5}
  |\ffs|\leq|\Bv||\ff|\leq\gC
  \end{alignat*}
  and inherits a form of quasiconformality:
    \begin{alignat*}{5} |\pLam\ffs|^2 &= |(\pLam\Bv)\ff+\Bv\pLam\ff|^2 
  \\&\leq \gC ( e^{-2\polaeps\lrv}+|\pLam\ff|^2 ) 
  \\&\leq \gC ( e^{-2\polaeps\lrv} + \Det\pLam\ff )
  \\&= \gC \Big( e^{-2\polaeps\lrv} + \Det\big(\Bv^{-1}\pLam\ffs+\subeq{\pLam(\Bv^{-1})}{\leq\gC e^{-\polaeps\lrv}}\subeq{\ffs}{\leq\gC}\big) \Big)
  \intertext{(continuity of $\Det$ as a function of the matrix)}
  &\leq \gC ( e^{-2\polaeps\lrv} + \Det\pLam\ffs)
  \end{alignat*} 
  for $\lrvmin<\infty$ sufficiently larger.

  If the boundaries were straight (or even periodic)
  the exponential would be absent;
  the reader may wish to ignore it at first reading.

  \paragraph{Dirichlet integral}

  Integrate over a domain $\carect$ that is a cartesian rectangle in $\lrv,\lam$ coordinates:
  \begin{alignat}{5} \int_\carect |\pLam\ffs|^2 - \gC e^{-2\polaeps\lrv} d\Lam 
  &\leq \gC
  \int_\carect \Det \frac{\partial(\fz,\fo)}{\partial(\lrv,\lam)} d\lrv\wedge d\lam
  \notag\\&\leq
  \gC  \int_\carect \subeq{d\fz \wedge d\fo}{=d(\fz d\fo)=d(-\fo d\fz)} \label{eq:fork0}
  \\&=
  \gC  \int_{\bdry \carect} \fz d\fo 
  \label{eq:forka}
  \intertext{(use $\fz=0$ or $\fo=0$ on $\Sigb$)}
  \label{eq:fork1}
  &=
  \gC  \int_{\Sigi} \fz d\fo 
  \intertext{(Cauchy-Schwarz inequality)}
  &\leq 
  \gC
  \Big( \int_{\Sigi} |\fz|^2 ds \Big)^{\half} \Big( \int_{\Sigi} |\ps\fo|^2 ds \Big)^{\half} 
  \label{eq:fork2}
  \intertext{(use $|\fz|\leq\gC$)}
  &\leq 
  \gC |\Sigi|^{\half} \Big(\int_{\Sigi} |\pLam\ffs|^2 ds \Big)^{\half}. \notag\end{alignat} 
  It is also necessary to obtain this estimate with exponent $1$ instead of $\half$, as follows. Case 1: if $\bdry\carect$ does not meet $\Slipb$, then we use $\int_{\bdry\carect}d\fo=0$ in \eqref{eq:forka} to insert
  an arbitrary constant $c$:
  \begin{alignat*}{5}
  ...
  &=
  \gC  \int_{\bdry \carect} (\fz-c) d\fo 
  \intertext{(Cauchy-Schwarz inequality)}
  &\leq 
  \gC \Big( \int_{\bdry \carect} (\fz-c)^2 d\fo \Big)^{\half} \Big( \int_{\bdry \carect} |\ps\fo|^2 ds \Big)^{\half}
  \intertext{(Poincar\'e-type inequality for the first integral, minimizing over $c$)}
  &\leq 
  \gC |\bdry\carect| \int_{\bdry \carect} |\pLam\ffs|^2 ds .
  \end{alignat*}
  Case 2: if the rectangle $\bdry\carect$ meets the $\lam=0$ side of $\Sigb$, 
  then $\fz=0$ at both endpoints of $\Sigi$ if the $\lam=1$ side is not met,
  or at one endpoint of each of the two line segments composing $\Sigi$ if the $\lam=1$ side is also met,
  so we may use another Poincar\'e-type inequality to estimate \eqref{eq:fork2} as 
  \begin{alignat*}{5} 
  ... 
  \leq
  \gC |\Sigi| \int_{\Sigi} |\pLam\ffs|^2 ds .
  \end{alignat*} 
  Case 3: if $\bdry\carect$ meets only the $\lam=1$ side, then we may use $d(-\fo d\fz)$
  instead of $d(\fz d\fo)$ from \eqref{eq:fork0} onward to obtain an analogous estimate.
  Altogether we have obtained
  \begin{alignat}{5} 
  \int_\carect |\pLam\ffs|^2 - \gC e^{-2\polaeps\lrv} d\Lam 
  &\leq
  \gC 
  \big( |\Sigi| \int_{\Sigi} |\pLam\ffs|^2 ds \big)^{1/\bx} \quad(\bx\in\set{1,2}).
  \label{eq:allcombined}
  \end{alignat}

  \paragraph{Boundedness and decay}

  For 
  \begin{alignat*}{5} \carect \defeq \cli{\lrvlo}{\lrvhi} \cartp \cli01  \quad(\lrvmin \leq \lrvlo \leq \lrvhi < \infty) \end{alignat*} 
  and
  \begin{alignat*}{5} J(\lrv) \defeq \int_0^1 |\pLam\ffs(\lrv,\lam)|^2 d\lam  \end{alignat*}
  \eqref{eq:allcombined} specializes for each $\bx$ to 
  \begin{alignat}{5}
  \supeq{ \int_{\lrvlo}^{\lrvhi} J ~d\lrv }{\eqdef \Dir}
  - 
    \gC
    \int_{\lrvlo}^{\lrvhi} e^{-2\polaeps\lrv} d\lrv
  &\leq 
  \gC
  \big(
  J(\lrvlo)^{1/\bx}
  + 
  J(\lrvhi)^{1/\bx} 
  \big)
  \label{eq:dirdir} 
  \\\impl\quad
  \Dir 
  - 
  \subeq{ 
    \gC
    \big( \int_{\lrvlo}^{\infty} e^{-2\polaeps\lrv} d\lrv + J(\lrvlo)^{1/\bx} \big)
  }{=T=\text{constant in $\lrvhi$}}
  &\leq
  (\frac{\partial\Dir}{\partial\lrvhi})^{1/\bx}
  =
  (\frac{\partial(\Dir-T)}{\partial\lrvhi})^{1/\bx}
  \notag
  \end{alignat}
  If $\Dir>T$ at some $\lrvhi$, and if $\bx>1$, then we have a classical case of blowup at a finite larger $\lrvhi$. 
  But $\pLam\ffs$ and hence $\Dir$ are defined and smooth for all $\lrvhi$ up to $\infty$ --- contradiction.

  Hence $\Dir$ is bounded as $\lrvhi\upconv\infty$,
  and since $J\geq 0$ that means the $\liminf$ of $J$ must be $0$. 
  Hence we may take that $\liminf$ in \eqref{eq:dirdir} to obtain
  \begin{alignat*}{5}
  \supeq{\int_{\lrvlo}^\infty J(\lrv)d\lrv}{\eqdef\Diri}
  - 
  \gC
  \int_{\lrvlo}^{\infty} e^{-2\polaeps\lrv} d\lrv
  &\leq 
  \gC
  J(\lrvlo)^{1/\bx}
  =
  (\frac{\partial\Diri}{\partial(-\lrvlo)})^{1/\bx}
  \\
  \impl\quad
  \Diri
  - 
  \subeq{ 
    \gC
    \int_{\lrvmin}^{\infty} e^{-2\polaeps\lrv} d\lrv
  }{=T_2=\text{const in $\lrvhi$}}
  &\leq 
  \gC
  (\frac{\partial(\Diri-T_2)}{\partial(-\lrvlo)})^{1/\bx}
  \end{alignat*}
  Again if $\bx>1$, namely $\bx=2$, if $\Diri>T_2$ at some $\lrvlo$ then we have blowup as $\lrvlo$ \emph{de}creases,
  except if the domain ends before then, at $\lrvlo=\lrvmin$.
  Integrating from a finite $\Diri$ there to larger $\lrvlo$ we obtain
  \begin{alignat*}{5} \Diri \leq \max\{T_2,\gC\frac1{\lrv-\lrvmin}\}. \end{alignat*} 
  In particular, after increasing $\lrvmin$ by an amount $\leq\gC$ for the remainder of the proof, we may use
  \begin{alignat*}{5} \Diri &\leq T_2 \leq \gC \qquad (\lrvmin\leq\lrvlo<\infty). \end{alignat*} 
  (The earlier bound $T$ was dependent on $J$, unlike $T_2$.)

  Finally, for $\bx=1$ instead, using $\Diri\leq\gC$ at $\lrvlo=\lrvmin$ we obtain
  \begin{alignat}{5} \Diri &\leq \gC e^{-2\morex\lrvlo}  \qquad (\lrvmin\leq\lrvlo<\infty)    \label{eq:Diri-est}\end{alignat} 
  where $2\morex$ depends (through $\gC,\polaeps,\lrvmin$) only on $\Dom,\isenc,\esssup\Mach$.

\paragraph{Dirichlet integrals on squares}

  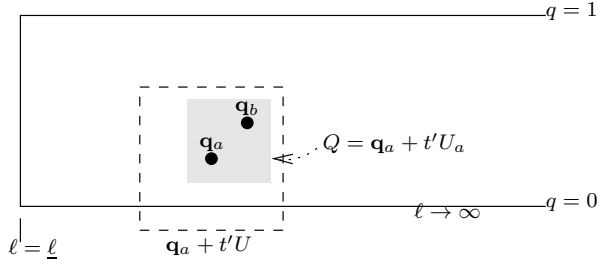
\begin{figure}
    \input{slidingrect.pstex_t}
    \caption{Comparing averages by Morrey estimates}
    \label{fig:slidingrect}
  \end{figure}

  Now we consider integral averages over expanding squares (see fig.\ \ref{fig:slidingrect}).
  Let 
  \begin{alignat*}{5} \Lama &= (\lrva,\lama)\in\roi{1+\lrvmin}{\infty}\cartp\cli01, \\
  \finrect &= \cli{-1}{1}\cartp\cli{-1}{1} \end{alignat*}
  As $t$ increases from $0$ to $1$, $\Lama+t\finrect$ starts to intersect with the $\lam=0,1$ slip boundary at some $t_\lam$; set
  \begin{alignat*}{5} \carect &= (\Lama+t\finrect)\isect(\roi{\lrvmin}{\infty}\cartp\cli01). \end{alignat*}
  \eqref{eq:allcombined} yields\footnote{Note that $\bdry\carect$ ``moves'' with ``normal speed'' $1$ by choice of $\finrect$.}
  \begin{alignat}{5} 
  \Dirt(t)
  \defeq
  \int_{\carect} |\pLam\ffs|^2 d\Lam 
  &\leq
  \gC 
  \big(
  \subeq{|\carect|}{\leq\gC t^2} e^{-2\polaeps\lrva}
  +
  \subeq{|\Sigi|}{\leq\gC t}
  \int_{\Sigi} |\pLam\ffs|^2 ds 
  \big)
  \notag\\&\leq
  \gC
  t^2 e^{-2\polaeps\lrva}
  +
  \gC t \pd t \Dirt(t)
  \label{eq:Dirt-ineq}
  \end{alignat} 
  for all $t\in\loi01$ 
  except possibly at $t_\lam$, where $\Dirt$ is at least Lipschitz,
  so we may integrate \eqref{eq:Dirt-ineq} to obtain
  (with $2\morex\in\boi01$ determined from $\polaeps,\gC$ and the $\morex$ in \eqref{eq:Diri-est}): 
  \begin{alignat*}{5} \Dirt(t) 
  &\leq 
  t^{2\morex} \cdot \big( \subeq{ \Dirt(1)}{\leq\Diri(\lrva-1)} + \gC e^{-2\morex\lrva} \big)
  \\&  \topref{eq:Diri-est}{\leq} 
  \gC t^{2\morex} e^{-2\morex\lrva} 
  \quad\text{for $0<t\leq1$.} 
  \end{alignat*} 
  Any square $\QQ\subset\roi{\lrvmin}{\infty}\cartp\cli01$ of sidelength $t\in\loi01$ is contained in one of these $\carect$, so
  \begin{alignat*}{5} \int_{\QQ} |\pLam\ffs|^2 d\Lam   \leq \gC t^{2\morex} e^{-2\morex\lrva}
  \end{alignat*}

  \paragraph{H\"older continuity}

  Let $\Lama\in\QQ$ arbitrary; since $\QQ$ has sidelength $t$, we may write $\QQ=\Lama+t\TT$ for a square $\TT\subset\finrect$.
  Then for $t'\in\cli0t$ the last inequality yields
  \begin{alignat}{5} \int_{\Lama+t'\TT} |\pLam\ffs|^2 d\Lam   \leq \gC (t')^{2\morex} e^{-2\morex\lrva} 
  \label{eq:QQestb}
  \end{alignat} 
  We use this estimate to compare integral averages $\avgint_\QQ=|\QQ|^{-1}\int_\QQ$: 
  \begin{alignat*}{5} 
  \Big| \avgint_{\Lama+t\TT} \ffs(\Lam) d\Lam - \ffs(\Lama) \Big|
  &=
  \Big| \avgint_{\TT} \ffs(\Lama+t\vu) - \ffs(\Lama+0\vu) d\vu \Big|
  \\&\leq
  \int_0^t \Big| \avgint_{\TT } \pLam\ffs(\Lama+t'\vu)\dotp\vu ~d\vu\Big|dt'
  \\&\leq
  \int_0^t
  \Big( \avgint_{\TT } |\pLam\ffs(\Lama+t'\vu)|^2 \Big)^{\half}
  \subeq{\Big( \avgint_{\TT } |\vu|^2 ~d\vu\Big)^{\half}}{\leq\gC}
  dt'
  \\&\leq
  \gC
  \int_0^t
  \Big( \avgint_{\Lama+t'\TT} |\pLam\ffs(\Lam)|^2 d\Lam \Big)^{\half} 
  dt'
  \\&\leq
  \gC
  \int_0^t 
  \Big( (t')^{-2} \subeq{ \int_{\Lama+t'\TT } |\pLam\ffs(\Lam)|^2 d\Lam 
  }{
    \topref{eq:QQestb}{\leq} \gC (t')^{2\morex} e^{-2\morex\lrva} 
  }
  \Big)^{\half} 
  dt'
  \\&\leq
  \gC
  t^{\morex} e^{-\morex\lrva}
  \end{alignat*} 
  For any other $\Lamb\in\QQ$ we have the same estimate
  \begin{alignat*}{5} 
  \Big| \avgint_{\QQ } \ffs(\Lam) d\Lam - \ffs(\Lamb) \Big|
  &\leq
  \gC
  t^{\morex} e^{-\morex\lrvb}. 
  \end{alignat*}
  Hence, for any given $\Lama,\Lamb\in\roi{\lrvmin}{\infty}\cartp\cli01$ with distance $t=|\Lama-\Lamb|\leq 1$ 
  we can choose a containing $Q$ with sidelength $t$ to get 
  \begin{alignat*}{5} 
  |\ffs(\Lama)-\ffs(\Lamb)| 
  &\leq 
  \gC t^\morex ( e^{-\morex\lrva} + e^{-\morex\lrvb}) 
  \leq 
  \gC |\Lama-\Lamb|^\morex e^{-\morex\min\set{\lrva,\lrvb}} 
  \end{alignat*} 
  For distances above $1$, 
  using a chain of such squares a geometric series generalizes the estimate to the entire strip $\roi{\lrvmin}{\infty}\cartp\cli01$.
  In particular $\ffs(\lrv,\lam)$ converges as $\lrv\conv\infty$, 
  and conversion from $\Lam$ to $\xx$ yields the claim of Theorem \ref{th:inflocalmorrey}. 
\end{proof}

\section{Conclusion}
\label{section:conclusion}%

\begin{proof}[Proof of Theorem \ref{th:infangle}]
  Theorem \ref{th:inflocalmorrey} shows that $\nabla\stf(\xx)\conv\nabla\stf(\infty)$ as $|\xx|\conv\infty$. 
  The slip condition yields $\nabla\stf(\xx)\dotp\vs(\xx)=0$ on each side of $\Slipb$, 
  where $\vs$ are continuous unit tangent fields. 
  The limits of $\vs$ enclose an angle $\Pola\notin\set{0,\pi,2\pi}$, 
  so they are linearly independent; 
  hence the limit of the boundary conditions yields that $\nabla\stf(\infty)=0$. 
\end{proof}
\noindent The proof of Theorem \ref{th:infangle-incomp} is analogous, except that in the incompressible case
boundedness of the velocity must be assumed explicitly.

\section*{Acknowledgement}

This material is based upon work partially supported by the
National Science Foundation under Grant No.\ NSF DMS-1054115
and by Taiwan MOST grant 105-2115-M-001-007-MY3.

\input{proangle.bbl}

\end{document}

%% file: kj.pstex_t
\begin{picture}(0,0)%
\includegraphics{kj.pstex}%
\end{picture}%
\setlength{\unitlength}{3947sp}%
\begingroup\makeatletter\ifx\SetFigFont\undefined%
\gdef\SetFigFont#1#2#3#4#5{%
  \reset@font\fontsize{#1}{#2pt}%
  \fontfamily{#3}\fontseries{#4}\fontshape{#5}%
  \selectfont}%
\fi\endgroup%
\begin{picture}(2631,952)(858,-404)
\put(873,401){\makebox(0,0)[lb]{\smash{{\SetFigFont{9}{10.8}{\rmdefault}{\mddefault}{\updefault}{\color[rgb]{0,0,0}$v_\infty$}%
}}}}
\put(1651, 89){\makebox(0,0)[lb]{\smash{{\SetFigFont{8}{9.6}{\rmdefault}{\mddefault}{\updefault}{\color[rgb]{0,0,0}lift}%
}}}}
\put(2551,239){\makebox(0,0)[lb]{\smash{{\SetFigFont{9}{10.8}{\rmdefault}{\mddefault}{\updefault}{\color[rgb]{0,0,0}Protruding}%
}}}}
\put(2776, 89){\makebox(0,0)[lb]{\smash{{\SetFigFont{9}{10.8}{\rmdefault}{\mddefault}{\updefault}{\color[rgb]{0,0,0}corner}%
}}}}
\end{picture}%

%% file: smoothangle.pstex_t
\begin{picture}(0,0)%
\includegraphics{smoothangle.pstex}%
\end{picture}%
\setlength{\unitlength}{3947sp}%
\begingroup\makeatletter\ifx\SetFigFont\undefined%
\gdef\SetFigFont#1#2#3#4#5{%
  \reset@font\fontsize{#1}{#2pt}%
  \fontfamily{#3}\fontseries{#4}\fontshape{#5}%
  \selectfont}%
\fi\endgroup%
\begin{picture}(3024,1511)(64,-748)
\put(348, 81){\makebox(0,0)[lb]{\smash{{\SetFigFont{7}{8.4}{\rmdefault}{\mddefault}{\updefault}{\color[rgb]{0,0,0}$\graph\polahi$}%
}}}}
\put(2354,479){\makebox(0,0)[lb]{\smash{{\SetFigFont{7}{8.4}{\rmdefault}{\mddefault}{\updefault}{\color[rgb]{0,0,0}$\rmin$}%
}}}}
\put(1388,-149){\makebox(0,0)[lb]{\smash{{\SetFigFont{7}{8.4}{\rmdefault}{\mddefault}{\updefault}{\color[rgb]{0,0,0}boundary}%
}}}}
\put(2576,-454){\makebox(0,0)[lb]{\smash{{\SetFigFont{7}{8.4}{\rmdefault}{\mddefault}{\updefault}{\color[rgb]{0,0,0}$\graph\polalo$}%
}}}}
\put(1388, 14){\makebox(0,0)[lb]{\smash{{\SetFigFont{7}{8.4}{\rmdefault}{\mddefault}{\updefault}{\color[rgb]{0,0,0}slip}%
}}}}
\end{picture}%

%% file: smoothanglesheet.pstex_t
\begin{picture}(0,0)%
\includegraphics{smoothanglesheet.pstex}%
\end{picture}%
\setlength{\unitlength}{3947sp}%
\begingroup\makeatletter\ifx\SetFigFont\undefined%
\gdef\SetFigFont#1#2#3#4#5{%
  \reset@font\fontsize{#1}{#2pt}%
  \fontfamily{#3}\fontseries{#4}\fontshape{#5}%
  \selectfont}%
\fi\endgroup%
\begin{picture}(2784,1721)(64,-973)
\put(1571,-478){\makebox(0,0)[lb]{\smash{{\SetFigFont{6}{7.2}{\rmdefault}{\mddefault}{\updefault}{\color[rgb]{0,0,0}boundary}%
}}}}
\put( 84,-221){\makebox(0,0)[lb]{\smash{{\SetFigFont{6}{7.2}{\rmdefault}{\mddefault}{\updefault}{\color[rgb]{0,0,0}Upstream}%
}}}}
\put( 84,-371){\makebox(0,0)[lb]{\smash{{\SetFigFont{6}{7.2}{\rmdefault}{\mddefault}{\updefault}{\color[rgb]{0,0,0}wall}%
}}}}
\put(191,649){\makebox(0,0)[lb]{\smash{{\SetFigFont{6}{7.2}{\rmdefault}{\mddefault}{\updefault}{\color[rgb]{0,0,0}$\vvi$}%
}}}}
\put(1974,-41){\makebox(0,0)[lb]{\smash{{\SetFigFont{6}{7.2}{\rmdefault}{\mddefault}{\updefault}{\color[rgb]{0,0,0}Vortex sheet}%
}}}}
\put(2434,-329){\makebox(0,0)[lb]{\smash{{\SetFigFont{6}{7.2}{\rmdefault}{\mddefault}{\updefault}{\color[rgb]{0,0,0}$\vv=0$}%
}}}}
\put(1571,-329){\makebox(0,0)[lb]{\smash{{\SetFigFont{6}{7.2}{\rmdefault}{\mddefault}{\updefault}{\color[rgb]{0,0,0}slip}%
}}}}
\end{picture}%

%% file: slidingrect.pstex_t
\begin{picture}(0,0)%
\includegraphics{slidingrect.pstex}%
\end{picture}%
\setlength{\unitlength}{3947sp}%
\begingroup\makeatletter\ifx\SetFigFont\undefined%
\gdef\SetFigFont#1#2#3#4#5{%
  \reset@font\fontsize{#1}{#2pt}%
  \fontfamily{#3}\fontseries{#4}\fontshape{#5}%
  \selectfont}%
\fi\endgroup%
\begin{picture}(3405,1678)(-89,-716)
\put(3301,839){\makebox(0,0)[lb]{\smash{{\SetFigFont{8}{9.6}{\rmdefault}{\mddefault}{\updefault}{\color[rgb]{0,0,0}$\lam=1$}%
}}}}
\put(3301,-361){\makebox(0,0)[lb]{\smash{{\SetFigFont{8}{9.6}{\rmdefault}{\mddefault}{\updefault}{\color[rgb]{0,0,0}$\lam=0$}%
}}}}
\put(2476,-436){\makebox(0,0)[lb]{\smash{{\SetFigFont{8}{9.6}{\rmdefault}{\mddefault}{\updefault}{\color[rgb]{0,0,0}$\lrv\conv\infty$}%
}}}}
\put(1126, 14){\makebox(0,0)[lb]{\smash{{\SetFigFont{8}{9.6}{\rmdefault}{\mddefault}{\updefault}{\color[rgb]{0,0,0}$\Lama$}%
}}}}
\put(1351,239){\makebox(0,0)[lb]{\smash{{\SetFigFont{8}{9.6}{\rmdefault}{\mddefault}{\updefault}{\color[rgb]{0,0,0}$\Lamb$}%
}}}}
\put(1901, -2){\makebox(0,0)[lb]{\smash{{\SetFigFont{8}{9.6}{\rmdefault}{\mddefault}{\updefault}{\color[rgb]{0,0,0}$\QQ=\Lama+t'\TT$}%
}}}}
\put(-74,-661){\makebox(0,0)[lb]{\smash{{\SetFigFont{8}{9.6}{\rmdefault}{\mddefault}{\updefault}{\color[rgb]{0,0,0}$\lrv=\lrvmin$}%
}}}}
\put(916,-642){\makebox(0,0)[lb]{\smash{{\SetFigFont{8}{9.6}{\rmdefault}{\mddefault}{\updefault}{\color[rgb]{0,0,0}$\Lama+t'\finrect$}%
}}}}
\end{picture}%

%% file: proangle.bbl
\providecommand{\bysame}{\leavevmode\hbox to3em{\hrulefill}\thinspace}
\providecommand{\MR}{\relax\ifhmode\unskip\space\fi MR }
\providecommand{\MRhref}[2]{%
  \href{http://www.ams.org/mathscinet-getitem?mr=#1}{#2}
}
\providecommand{\href}[2]{#2}